\documentclass[oneside,english]{amsart}
\usepackage{lmodern}
\usepackage[T1]{fontenc}
\usepackage[latin9]{inputenc}
\usepackage{geometry}
\usepackage{color}
\usepackage{babel}
\usepackage{units}
\usepackage{textcomp}
\usepackage{enumitem}
\usepackage{amsbsy}
\usepackage{amstext}
\usepackage{amsthm}
\usepackage{amssymb}
\usepackage{stmaryrd}
\usepackage{graphicx}
\usepackage[unicode=true,pdfusetitle,
 bookmarks=false,
 breaklinks=false,pdfborder={0 0 0},pdfborderstyle={},backref=false,colorlinks=true]
 {hyperref}
\hypersetup{
 linkcolor=blue,  citecolor=blue, urlcolor=blue}

\makeatletter

\newcommand{\noun}[1]{\textsc{#1}}

\numberwithin{equation}{section}
\numberwithin{figure}{section}
\theoremstyle{remark}
\newtheorem*{acknowledgement*}{\protect\acknowledgementname}
\theoremstyle{plain}
\newtheorem{thm}{\protect\theoremname}[section]
\theoremstyle{definition}
\newtheorem{defn}[thm]{\protect\definitionname}
\theoremstyle{plain}
\newtheorem{prop}[thm]{\protect\propositionname}
\theoremstyle{remark}
\newtheorem{rem}[thm]{\protect\remarkname}

\usepackage[abbrev,alphabetic]{amsrefs}

\usepackage{graphicx}
\usepackage{setspace}
\usepackage{url}
\usepackage[perpage,symbol]{footmisc}
\usepackage{multicol}
\usepackage{etoolbox}
\usepackage{bbm}
\usepackage{enumitem}
\usepackage{caption}
\usepackage[capbesideposition={left,center}]{floatrow}
\usepackage{etoolbox}

\date{}
\DefineFNsymbols*{lamportnostar}[math]{{(\dagger)}{(\ddagger)}{(\S)}{(\P)}{(\|)}{(\dagger\dagger)}{(\ddagger\ddagger)}}
\setfnsymbol{lamportnostar}

\DeclareMathOperator{\supp}{supp}

\DeclareMathOperator{\Spec}{Spec}

\DeclareMathOperator{\diam}{diam}

\DeclareMathOperator{\tr}{tr}

\newcommand{\one}{\mathbbm{1}}
\newcommand{\Mod}[1]{\,\left(\textup{mod}\;#1\right)}

\theoremstyle{plain}
\newtheorem{claim}[thm]{\protect\claimname}

\theoremstyle{plain}

\theoremstyle{definition}

\newtheorem{rem}[thm]{Remark}
\newtheorem*{rems*}{Remarks}
\newtheorem*{disc*}{Discussion}

\allowdisplaybreaks[2]

\makeatother

\providecommand{\acknowledgementname}{Acknowledgement}
\providecommand{\definitionname}{Definition}
\providecommand{\propositionname}{Proposition}
\providecommand{\remarkname}{Remark}
\providecommand{\theoremname}{Theorem}

\begin{document}
\title{Ramanujan Graphs and Digraphs}
\author{Ori Parzanchevski}
\thanks{Supported by ISF grant 1031/17.}
\begin{abstract}
Ramanujan graphs have fascinating properties and history. In this
paper we explore a parallel notion of Ramanujan \emph{digraphs}, collecting
relevant results from old and recent papers, and proving some new
ones. \emph{Almost-normal }Ramanujan digraphs are shown to be of special
interest, as they are extreme in the sense of an Alon-Boppana theorem,
and they have remarkable combinatorial features, such as small diameter,
Chernoff bound for sampling, optimal covering time and sharp cutoff.
Other topics explored are the connection to Cayley graphs and digraphs,
the spectral radius of universal covers, Alon's conjecture for random
digraphs, and explicit constructions of almost-normal Ramanujan digraphs.
\end{abstract}

\maketitle

\section{Introduction}

A connected $k$-regular graph is called a \emph{Ramanujan graph}
if every eigenvalue $\lambda$ of its adjacency matrix (see definitions
below) satisfies either 
\begin{align*}
~ &  &  & \left|\lambda\right|=k,\quad\text{or}\quad\left|\lambda\right|\leq2\sqrt{k-1}. &  &  & \left(Ramanujan~graph\right)
\end{align*}
While the generalized Ramanujan conjecture appears in the first constructions
of such graphs \cite{LPS88,margulis1988explicit}, the reason that
lead Lubotzky, Phillips and Sarnak to coin the term Ramanujan graphs
is that by their very definition, they present the phenomenon of extremal
spectral behavior, which Ramanujan observed in a rather different
setting.

In the case of graphs, this can be stated in two ways: Ramanujan graphs
spectrally mimic their universal cover, the infinite $k$-regular
tree $\mathcal{T}_{k}$, whose spectrum is the interval 
\[
\Spec\left(\mathcal{T}_{k}\right)=\left[-2\sqrt{k-1},2\sqrt{k-1}\right]
\]
\cite{Kesten1959}; And, they are asymptotically optimal: the Alon-Boppana
theorem (cf.\ \cite{LPS88,Nil91,HLW06}) states that for any $\varepsilon>0$,
there is no infinite family of $k$-regular graphs for which all nontrivial
adjacency eigenvalues satisfy $\left|\lambda\right|\leq2\sqrt{k-1}-\varepsilon$
(the trivial eigenvalues are by definition $\pm k$). These two observations
are closely related - in fact, any infinite family of quotients of
a common covering graph $\widetilde{\mathcal{G}}$ cannot ``do better''
than $\widetilde{\mathcal{G}}$ (see \cite{greenberg1995spectrum,grigorchuk1999asymptotic}
for precise statements).

A major interest in the adjacency spectrum of a graph comes from the
notion of \emph{expanders} - graphs of bounded degree whose nontrivial
adjacency spectrum is of small magnitude. Such graphs have strong
connectedness properties which are extremely useful: see \cite{Lub10,HLW06,Lub12}
for extensive surveys on properties of expanders in mathematics and
computer science.

Ramanujan graphs, which stand out as the optimal expanders (from the
spectral point of view), have a rich theory and history. The purpose
of this paper is to suggest that a parallel theory should be developed
for directed graphs (\emph{digraphs}, for short), where by a \emph{Ramanujan
digraph} we mean a $k$-regular digraph whose adjacency eigenvalues
satisfy either 
\begin{align}
~ &  &  & \left|\lambda\right|=k,\quad\text{or}\quad\left|\lambda\right|\leq\sqrt{k},\phantom{-1} &  &  & \left(Ramanujan~digraph\right)\label{eq:ram-digraph}
\end{align}

where the reasons for this definition will be made clear along the
paper. 

The idea of ``Ramanujan digraphs'' arose during the work on the
papers \cite{Parzanchevski2018SuperGoldenGates,Lubetzky2017RandomWalks};
While we believe that the term itself is new, several classic results
can be interpreted as saying something about these graphs (see for
example §\ref{subsec:Directed-line-graphs} and §\ref{subsec:Universal-Cayley-graphs}).
We survey here both classic results and ones from the mentioned papers,
and prove several new ones. We remark that for the most part of the
paper we focus on finite graphs and digraphs, with infinite ones appearing
mainly as universal covers. Without doubt, they merit further study
in their own right (see also §\ref{sec:Questions}).

The paper unfolds as follows: After giving the definitions in §\ref{sec:Definitions}
and various examples in §\ref{sec:Examples}, we prove that there
are very few normal Ramanujan digraphs in §\ref{subsec:Normal-Ramanujan-digraphs}.
We then turn to almost-normal digraphs in §\ref{sec:almost-normal},
proving an Alon-Boppana type theorem, and surveying their spectral
and combinatorial features, such as optimal covering, sharp cutoff,
small diameter and Chernoff sampling bound. We then explore Ramanujan
digraphs from the perspective of universal covers §\ref{subsec:Universal-Object},
and infinite Cayley graphs §\ref{subsec:Universal-Cayley-graphs}.
In §\ref{subsec:Explicit-constructions} we discuss an explicit construction
of Ramanujan digraphs as Cayley graphs of finite groups, which is
similar to the LPS construction \cite{LPS88}, but applies to any
$PGL_{d}$ and not only to $PGL_{2}$. In §\ref{subsec:Riemann-Hypothesis}
we touch upon zeta functions and the Riemann Hypothesis, in §\ref{subsec:Alon-conjecture}
we discuss Alon's second eigenvalue conjecture for digraphs, and finally
in §\ref{sec:Questions} we present some questions.
\begin{acknowledgement*}
We would like to express our gratitude to Noga Alon, Amitay Kamber,
Alex Lubotzky, Doron Puder and Alain Valette for various helpful remarks
and suggestions.
\end{acknowledgement*}

\section{\label{sec:Definitions}Definitions}

Throughout the paper we denote by $\mathcal{G}=\left(V_{\mathcal{G}},E_{\mathcal{G}}\right)$
a connected $k$-regular graph on $n$ vertices, where by a \emph{graph}
we always mean an undirected one. Its adjacency matrix $A=A_{\mathcal{G}}$,
indexed by $V$, is defined by $A_{v,w}=1$ if $v\sim w$ ($v$ and
$w$ are neighbors in $\mathcal{G}$), and $0$ otherwise\footnote{On occasions we allow loops and multiple edges, in which case $A_{v,w}$
is the number of edges between $v$ and $w$.}. Since $\sim$ is symmetric, so is $A$, hence it is self adjoint
with real spectrum. The constant function $\one$ is an eigenvector
of $A$ with eigenvalue $k$, and when $\mathcal{G}$ is bipartite,
say $V=L\amalg R$, the function $\one_{L}-\one_{R}$ is an eigenvector
with eigenvalue $-k$. We call these eigenvalues and eigenvectors
\emph{trivial}, and denote by $L_{0}^{2}=L_{0}^{2}\left(V\right)$
their orthogonal complement in $L^{2}\left(V\right)$, namely
\[
L_{0}^{2}\left(V\right)=\begin{cases}
\one^{\bot} & \mathcal{G}\text{ is not bipartite}\\
\left\langle \one_{L},\one_{R}\right\rangle ^{\bot} & \mathcal{G}\text{ is bipartite}.
\end{cases}
\]
Observe that $A$ restricts to a self-adjoint operator on $L_{0}^{2}\left(V\right)$,
and recall that for self-adjoint (and even normal) operators, the
spectral radius
\[
\rho\left(M\right)=\max\left\{ \left|\lambda\right|\,\big|\,\lambda\in\Spec\left(M\right)\right\} 
\]
coincides with the operator norm 
\[
\left\Vert M\right\Vert =\max_{v\neq0}\frac{\left\Vert Mv\right\Vert }{\left\Vert v\right\Vert }.
\]

\begin{defn}[\cite{LPS88}]
A $k$-regular graph $\mathcal{G}$ is a \emph{Ramanujan graph} if
\[
\rho\left(\mathcal{G}\right)\overset{{\scriptscriptstyle def}}{=}\rho\left(A_{\mathcal{G}}\big|_{L_{0}^{2}}\right)=\left\Vert A_{\mathcal{G}}\big|_{L_{0}^{2}}\right\Vert \leq2\sqrt{k-1}.
\]
\end{defn}

Moving on to digraphs, we denote by $\mathcal{D}$ a finite connected
$k$-regular directed graph, by which we mean that each vertex has
$k$ incoming and $k$ outgoing edges. Now, $A_{v,w}=1$ whenever
$v\rightarrow w$ (namely, there is an edge from $v$ to $w$) and
since $A$ is no longer symmetric, its spectrum is not necessarily
real. However, by regularity we still have
\[
\rho\left(A\right)=\left\Vert A\right\Vert _{1}=\left\Vert A\right\Vert _{2}=\left\Vert A\right\Vert _{\infty}=k,
\]
as any square matrix satisfies
\begin{equation}
\left\Vert A\right\Vert _{2}^{2}=\rho\left(A^{*}A\right)\leq\left\Vert A^{*}A\right\Vert _{\infty}\leq\left\Vert A^{*}\right\Vert _{\infty}\left\Vert A\right\Vert _{\infty}=\left\Vert A\right\Vert _{1}\left\Vert A\right\Vert _{\infty},\label{eq:Norm2-1-infty}
\end{equation}
and $\one$ is still a $k$-eigenfunction. If $\mathcal{D}$ is $m$-periodic,
namely $V_{\mathcal{D}}=\coprod_{j=0}^{m-1}V_{j}$ with every edge
starting in $V_{j}$ terminating in $V_{\left(j+1\!\mod m\right)}$,
then $e^{2\pi ti/m}k$ are also eigenvalues (with $t=1,\ldots,m\!-\!1$),
with corresponding eigenfunctions $\sum_{j=0}^{m-1}e^{2\pi jti/m}\one_{V_{j}}$.
By Perron-Frobenius theory, all eigenvalues of absolute value $k$
arise in this manner. We call these eigenvalues (including $k$) trivial,
and denote by $L_{0}^{2}$ the orthogonal complement of their eigenfunctions
in $L^{2}\left(V_{\mathcal{D}}\right)$. Even though $A$ is not self-adjoint
or normal, the regularity assumptions ensures that it still restricts
to $L_{0}^{2}$, and we make the following definition:
\begin{defn}
\label{def:digraph}A $k$-regular digraph $\mathcal{D}$ is a \emph{Ramanujan
digraph} if 
\[
\rho\left(\mathcal{D}\right)\overset{{\scriptscriptstyle def}}{=}\rho\left(A_{\mathcal{D}}\big|_{L_{0}^{2}}\right)\leq\sqrt{k}.
\]
\end{defn}

The bound $\left\Vert A\big|_{L_{0}^{2}}\right\Vert \leq\sqrt{k}$
does not have to hold anymore; Indeed, we will see that there are
Ramanujan digraphs for which $\left\Vert A\big|_{L_{0}^{2}}\right\Vert =k$,
which is as bad as one can have for a $k$-regular adjacency operator
(in the undirected settings, this would mean that the graph is disconnected).
For spectral analysis, the operator norm is much more important than
the spectral radius, and this is what makes digraphs harder to study
than graphs.

We say that a digraph $\mathcal{D}$ is self-adjoint, or normal, if
its adjacency matrix is. In these cases we do have $\left\Vert A_{\mathcal{D}}\big|_{L_{0}^{2}}\right\Vert =\rho\left(\mathcal{D}\right)$,
and much of the theory of expanders remains as it is for graphs (see
for example \cite{Vu2008Sumproductestimates}). However, we will see
in Proposition \ref{prop:fin-many-normal} that there are very few
normal Ramanujan digraphs. A main novelty of \cite{lubetzky2016cutoff},
which was developed further in \cite{Lubetzky2017RandomWalks}, is
the idea of \emph{almost-normal }digraphs:
\begin{defn}
A matrix is \emph{$r$-normal} if it is unitarily equivalent to a
block-diagonal matrix with blocks of size at most $r\times r$. A
digraph is called $r$-normal if its adjacency matrix is $r$-normal,
and a family of matrices (or digraphs) is said to be almost-normal
if its members are $r$-normal for some fixed $r<\infty$.

We shall see in §\ref{sec:almost-normal} that for many applications,
almost-normal digraphs are almost as good as normal ones.
\end{defn}

\section{\label{sec:Examples}Examples}

\subsection{\label{subsec:Complete-digraphs}Complete digraphs}

For $m,k\in\mathbb{N}$, we define the complete $k$-regular $m$-periodic
digraph $\mathcal{K}_{k,m}$ by 
\begin{align*}
V_{\mathcal{K}_{k,m}} & =\left\{ \left(x,y\right)\,\middle|\,x\in\nicefrac{\mathbb{Z}}{m\mathbb{Z}},\ y\in\left[k\right]\right\} \\
E_{\mathcal{K}_{k,m}} & =\left\{ \left(x,y\right)\rightarrow\left(x+1,z\right)\,\middle|\,x\in\nicefrac{\mathbb{Z}}{m\mathbb{Z}},\ y,z\in\left[k\right]\right\} .
\end{align*}
This is a normal Ramanujan digraph on $n=km$ vertices, with $m$
trivial eigenvalues coming from periodicity, and $\left(k-1\right)m$
times the eigenvalue zero. This shows that one should focus on the
case of bounded degree and periodicity, for otherwise infinite families
of trivial examples arise.

\subsection{\label{subsec:Projective-planes}Projective planes and hyperplanes}

The \emph{Projective plane over $\mathbb{F}_{p}$} is the undirected
bipartite graph whose vertices represent the lines and planes in $\mathbb{F}_{p}^{3}$,
and whose edges correspond to the relation of inclusion. It is $k$-regular
for $k=p+1$, and has $n=2\left(p^{2}+p+1\right)$ vertices. Its nontrivial
spectrum is $\pm\sqrt{k-1}$ (each repeating $p^{2}+p$ times), so
it is Ramanujan. In fact, it is twice better than Ramanujan, which
only requires $\left|\lambda\right|\leq2\sqrt{k-1}$. We can therefore
consider it as a digraph, with each edge appearing with both directions,
and obtain a $k$-regular self-adjoint Ramanujan digraph, since the
adjacency matrix remains the same. 

More generally, the bipartite graph of lines against $d$-spaces in
$\mathbb{F}_{p}^{d+1}$ (with respect to inclusion) has $n=2\cdot\frac{p^{d+1}-1}{p-1}$
vertices and is $k$-regular with $k=\frac{p^{d}-1}{p-1}$. Its nontrivial
eigenvalues are $\pm\sqrt{p^{d-1}}=\pm\sqrt{k-\frac{p^{d-1}-1}{p-1}}$,
so we obtain a self-adjoint Ramanujan digraph for every $d$.

\subsection{\label{subsec:Paley-digraphs}Paley digraphs}

For a prime $p$ with $p\equiv3\Mod{4}$, the Paley digraph $\mathcal{PD}\left(p\right)$
\cite{graham1971constructive} has $V=\mathbb{F}_{p}$ and 
\[
E=\left\{ a\rightarrow b\,\middle|\,\left(\tfrac{b-a}{p}\right)=1\right\} ,
\]
where $\left(\frac{\cdot}{\cdot}\right)$ is the Legendre symbol.
It is a $k=\frac{p-1}{2}$-regular normal digraph, with nontrivial
eigenvalues $\frac{-1\pm i\sqrt{p}}{2}$ (this is a nice exercise
in Legendre symbols). These are of absolute value $\sqrt{\frac{k+1}{2}}$,
so $\mathcal{PD}\left(p\right)$ is a normal Ramanujan digraph.

\bigskip{}

It turns out that examples as in §\ref{subsec:Complete-digraphs}-§\ref{subsec:Paley-digraphs}
are limited. In §\ref{subsec:Normal-Ramanujan-digraphs} we will prove:
\begin{prop}
\label{prop:fin-many-normal}For any fixed $k\geq2$ and $m\geq1$
there are only finitely many $k$-regular $m$-periodic normal (and
in particular, self-adjoint) Ramanujan digraphs.
\end{prop}

Thus, if we wish to fix the regularity $k$ and periodicity $m$,
and yet take $\left|V\right|=n$ to infinity we must move on to non-normal
graphs.

\subsection{Extremal directed expanders}

The De Bruijn graph $\mathcal{DB}\left(k,s\right)$ is a $k$-regular
aperiodic Ramanujan digraph with 
\begin{align*}
V_{\mathcal{DB}\left(k,s\right)} & =\left[k\right]^{s}\qquad\text{(\ensuremath{\left[k\right]=\{1,\ldots,k\}}, so \ensuremath{n=k^{s}})}\\
E_{\mathcal{DB}\left(k,s\right)} & =\left\{ \left(a_{1},\ldots,a_{s}\right)\rightarrow\left(a_{2},\ldots,a_{s},t\right)\,\middle|\,a_{i},t\in\left[k\right]\right\} .
\end{align*}
Just as complete digraphs, the nontrivial spectrum of $\mathcal{DB}\left(k,s\right)$
consists entirely of zeros. However, its adjacency matrix is not diagonalizable,
and it has Jordan blocks of size $s$, so in particular, these do
not form an almost-normal family even for a fixed $k$. The Kautz
digraph is another example with similar properties.\footnote{For the spectrum of the symmetrization of De Bruin and Kautz digraphs,
see \cite{delorme1998spectrum}.} 

In \cite{Li1992Charactersumsand} Feng and Li show that $k$-regular
$r$-periodic \emph{diagonalizable} digraphs must have $\rho\left(\mathcal{D}\right)\geq1$
once $n>kr$. Furthermore, for any $n$ which is co-prime to $k$
they give an explicit construction of a $k$-regular $r$-periodic
digraph on $nr$ vertices with $\rho\left(\mathcal{D}\right)=1$.
\begin{rem}
De Bruijn graphs show that a direct analogue of the Alon-Boppana theorem
(with respect to any positive $\varepsilon$) does not hold for digraphs
in general. In §\ref{sec:almost-normal} we will see that in the settings
of almost-normal digraphs, an Alon-Boppana theorem does hold, with
the bound $\sqrt{k}$.
\end{rem}

\subsection{\label{subsec:Directed-line-graphs}Directed line graphs}

In this section we assume that $\mathcal{G}$ is a \textbf{$\left(\boldsymbol{k+1}\right)$-regular}
graph, and we define its $\boldsymbol{k}$\textbf{-regular} \emph{line-digraph}
$\mathcal{D}_{L}\left(\mathcal{G}\right)$ as follows: 
\begin{align*}
V_{\mathcal{D}_{L}\left(\mathcal{G}\right)} & =\left\{ \left(v,w\right)\,\middle|\,v,w\in V_{\mathcal{G}},\ v\sim w\right\} \\
E_{\mathcal{D}_{L}\left(\mathcal{G}\right)} & =\left\{ \left(v,w\right)\rightarrow\left(w,u\right)\,\middle|\,u\neq v\right\} .
\end{align*}
Namely, the vertices correspond to edges in $\mathcal{G}$ with a
chosen direction, and a $\mathcal{G}$-edge is connected to another
one in $\mathcal{D}_{L}\left(\mathcal{G}\right)$ if they form a non-backtracking
path of length 2 in $\mathcal{G}$. The importance of this construction
is that non-backtracking walks on $\mathcal{G}$ are encoded precisely
by regular (memory-less) walks on $\mathcal{D}_{L}\left(\mathcal{G}\right)$
(see Figure \ref{fig:The-line-digraph}).

\begin{figure}[h]
\begin{centering}
\includegraphics[width=0.6\columnwidth]{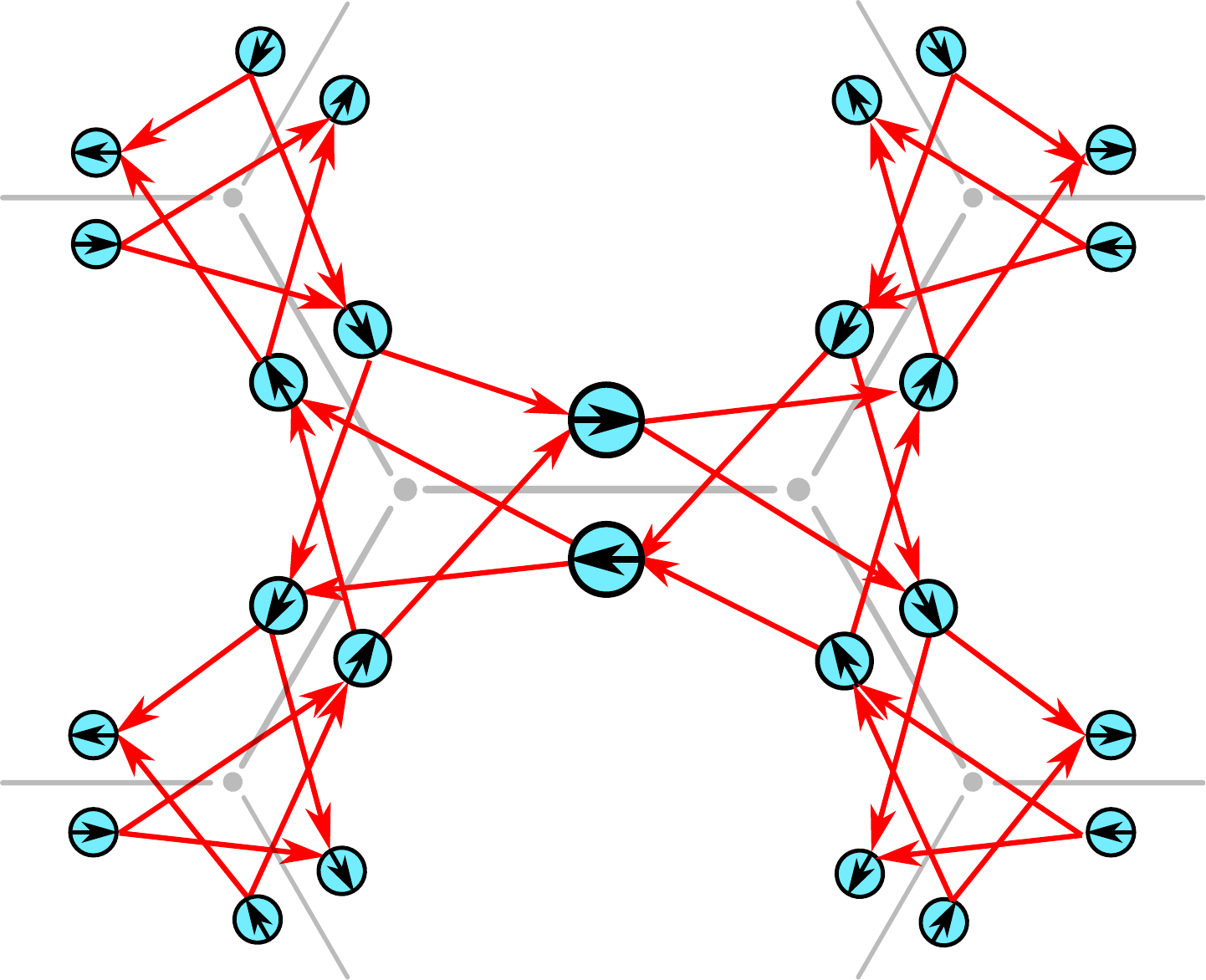}
\par\end{centering}
\caption{\label{fig:The-line-digraph}The local view of the line-digraph of
a 3-regular graph (the original graph is shown in the background).}
\end{figure}

By Hashimoto's interpretation of the Ihara-Bass formula (cf.\ \cite{ihara1966discrete,sunada1986functions,hashimoto1989zeta,bass1992ihara,stark1996zeta,foata1999combinatorial,kotani2000zeta,lubetzky2016cutoff}),
the spectra of $\mathcal{G}$ and $\mathcal{D}_{L}\left(\mathcal{G}\right)$
are related: 
\begin{equation}
\Spec\left(\mathcal{D}_{L}\left(\mathcal{G}\right)\right)=\left\{ \frac{\lambda\pm\sqrt{\lambda^{2}-4k}}{2}\,\middle|\,\lambda\in\Spec\mathcal{G}\right\} \cup\big\{\underbrace{\pm1\ ,\ \ldots\ ,\ \pm1}_{{\scriptscriptstyle \left|E_{\mathcal{G}}\right|-\left|V_{\mathcal{G}}\right|}\text{ times}}\big\}.\label{eq:spec-line-dig}
\end{equation}
One can easily check that 
\begin{equation}
\left|\lambda\right|\leq2\sqrt{\left(k+1\right)-1}\qquad\Longleftrightarrow\qquad\left|\tfrac{1}{2}\left(\lambda\pm\sqrt{\lambda^{2}-4k}\right)\right|=\sqrt{k},\label{eq:ram-to-dig-ram}
\end{equation}
so that $\mathcal{G}$ is a Ramanujan graph if and only if $\mathcal{D}_{L}\left(\mathcal{G}\right)$
is a Ramanujan digraph. Therefore, any construction of Ramanujan graphs
(e.g.\ \cite{LPS88,morgenstern1994existence,marcus2013interlacing})
can be used to construct Ramanujan digraphs.

The digraph $\mathcal{D}_{L}\left(\mathcal{G}\right)$ is not normal
(as one can easily verify by applying $AA^{T}$ and $A^{T}A$ to some
$\mathrel{\text{\ooalign{\ensuremath{\varbigcirc}\cr\ensuremath{\rightarrow}}}}$
in Figure \ref{fig:The-line-digraph}), and it turns out that the
singular values of $A$ are as bad as can be: the trivial singular
value $k$ repeats $\left|V_{\mathcal{G}}\right|$ times. This reflects
the fact that the walk described by $A^{T}A$ is highly disconnected:
the edges entering a fixed vertex form a connected component, since
\[
A^{T}A\left(v\rightarrow w\right)=A^{T}\left\{ \left(w\rightarrow u\right)\,\middle|\,{w\sim u\atop u\neq v}\right\} =\left\{ \left(u'\rightarrow w\right)\,\middle|\,w\sim u'\right\} 
\]
(this is easier to see in Figure \ref{fig:The-line-digraph} than
algebraically). In particular, this shows that $\left\Vert A\big|_{L_{0}^{2}\left(V_{\mathcal{D}_{L}\left(\mathcal{G}\right)}\right)}\right\Vert =k$.
The breakthrough in \cite{lubetzky2016cutoff} is the understanding
that $A$ is always 2-normal, and that this is good enough for the
analysis of the random walk on $\mathcal{D}_{L}\left(\mathcal{G}\right)$
(see §\ref{sec:almost-normal} below).

\subsection{\label{subsec:Collision-free-walks-on}Collision-free walks on affine
buildings}

In the previous example, a certain walk on the \emph{directed} edges
of an \emph{undirected} graph $\mathcal{G}$ gave rise to a digraph
$\mathcal{D}_{L}\left(\mathcal{G}\right)$, which was a Ramanujan
digraph whenever $\mathcal{G}$ was a Ramanujan graph. In \cite{Lubetzky2017RandomWalks}
this is generalized to higher dimension: considering some walk $W$
on the cells of a simplicial complex $\mathcal{X}$ (possibly oriented
or ordered cells), one asks when is the digraph $\mathcal{D}_{W}\left(\mathcal{X}\right)$
which represents this walk a Ramanujan digraph.

It turns out that the key is the following property: We say that a
digraph is \emph{collision-free }if it has at most one (directed)
path from any vertex $v$ to any vertex $w$. The digraph $\mathcal{D}_{L}\left(\mathcal{G}\right)$
from §\ref{subsec:Directed-line-graphs} is not collision-free - indeed,
a regular graph with this property must be infinite - but the line-digraph
of the universal cover of $\mathcal{G}$, namely $\mathcal{D}_{L}\left(\mathcal{T}_{k+1}\right)$,
is indeed collision-free: Two non-backtracking walkers which start
on the same directed edge on the tree will never reunite, once separated.
The main theorem in \cite{Lubetzky2017RandomWalks} is this:
\begin{thm}[\cite{Lubetzky2017RandomWalks}]
Let $\mathcal{X}$ be a complex whose universal cover is the affine
Bruhat-Tits building $\mathcal{B}$, and let $W$ be a geometric regular
random walk operator. If $W$ is collision-free on $\mathcal{B}$
(namely, $\mathcal{D}_{W}\left(\mathcal{B}\right)$ is collision-free),
and $\mathcal{X}$ is a Ramanujan complex, then $\mathcal{D}_{W}\left(\mathcal{X}\right)$
is a Ramanujan digraph.
\end{thm}

Here \emph{geometric} means that the random walk commutes with the
symmetries of $\mathcal{B}$; Properly defining the other terms in
the theorem will take us too far afield, and we refer the interested
reader to \cite{Lubotzky2005a,Lubotzky2013,Lubetzky2017RandomWalks,Lubotzky2020RamanujangraphsRamanujan}. 

Let us give one concrete example: the \emph{geodesic edge walk }on
a complex goes from a directed edge $\left(v,w\right)$ to the directed
edge $\left(w,u\right)$ if $u\neq v$ (no backtracking), and in addition
$\left\{ v,w,u\right\} $ is \textbf{not }a triangle in the complex
(so the path $v\rightarrow w\rightarrow u$ is not ``homotopic''
to the shorter path $v\rightarrow u$). 

The edges of the $d$-dimensional Bruhat-Tits building of type $\widetilde{A}_{d}$
are colored by $\left\{ 1,\ldots,d\right\} $ (loc.\ cit.), and the
geodesic walk restricted to edges of color 1 forms a regular collision-free
walk on the building. Thus, by the theorem above, the same walk on
Ramanujan complexes of type $\widetilde{A}_{d}$, as constructed in
\cite{li2004ramanujan,Lubotzky2005b,first2016ramanujan}, gives a
Ramanujan digraph. In the case $d=1$, the building $\widetilde{A}_{1}$
is a regular tree, its Ramanujan quotients are Ramanujan graphs, and
the geodesic edge walk is simply the non-backtracking walk, so we
obtain again the example from §\ref{subsec:Directed-line-graphs}.

Finally, all geometric walks on quotients of a fixed building $\mathcal{B}$
form a family of almost-normal digraphs \cite[Prop.\ 4.5]{Lubetzky2017RandomWalks}.
For the geodesic edge walk on $\widetilde{A}_{d}$-Ramanujan complexes,
the corresponding Ramanujan digraphs are sharply $\left(d+1\right)$-normal
\cite[Prop.\ 5.3, 5.4]{Lubetzky2017RandomWalks}, and they can be
made to be $m$-periodic for any $m$ dividing $\left(d+1\right)$.

\subsection{\label{subsec:Normal-Ramanujan-digraphs}Normal Ramanujan digraphs}

We now turn to the proof of Proposition \ref{prop:fin-many-normal},
for which we need a quantitative version of the Alon-Boppana theorem.
We use the following:
\begin{thm}[{\cite[Thm.\ 1 with $s=2$]{nilli2004tight}}]
\label{thm:Quant-Alon-Boppana}The second largest eigenvalue of a
$k$-regular graph $\mathcal{G}$ is at least $2\sqrt{k-1}\cos\left(\frac{2\pi}{\diam\mathcal{G}}\right)$.
\end{thm}

\begin{proof}[Proof of Proposition \ref{prop:fin-many-normal}]
Let $\mathcal{D}$ be a $k$-regular normal Ramanujan digraph on
$n$ vertices, and let $\mathcal{G}$ be its symmetrization, namely,
$A_{\mathcal{G}}=A_{\mathcal{D}}+A_{\mathcal{D}}^{T}$. Assume for
now that $\mathcal{D}$ is aperiodic. From normality of $A_{\mathcal{D}}$
we obtain 
\begin{equation}
\rho\left(\mathcal{G}\right)=\max\left\{ \lambda+\overline{\lambda}\,\middle|\,\lambda\in\Spec A_{\mathcal{D}}\big|_{L_{0}^{2}}\right\} \leq2\sqrt{k},\label{eq:symmetrized-rho}
\end{equation}
and we would like to combine this with Theorem \ref{thm:Quant-Alon-Boppana}.
For a $k_{\mathcal{G}}$-regular graph with $k_{\mathcal{G}}\geq4$,
Moore's bound \cite{hoffman1960moore} gives
\[
n\leq1+k_{\mathcal{G}}\sum_{j=1}^{\diam\mathcal{G}}\left(k_{\mathcal{G}}-1\right)^{j-1}\leq2\left(k_{\mathcal{G}}-1\right)^{\diam\mathcal{G}},
\]
so that Theorem \ref{thm:Quant-Alon-Boppana} implies (for $k_{\mathcal{G}}\geq4$)
\begin{equation}
\rho\left(\mathcal{G}\right)\geq2\sqrt{k_{\mathcal{G}}-1}\cos\left(\tfrac{2\pi}{\log_{k_{\mathcal{G}}-1}\left(n/2\right)}\right)\geq2\sqrt{k_{\mathcal{G}}-1}\left(1-\tfrac{2\pi^{2}}{\log_{k_{\mathcal{G}}-1}^{2}\left(n/2\right)}\right)\label{eq:quant-alon-boppana}
\end{equation}
Our $\mathcal{G}$ is $2k$-regular, so that (\ref{eq:symmetrized-rho})
and (\ref{eq:quant-alon-boppana}) combine to
\[
1-\tfrac{2\pi^{2}}{\log_{2k-1}^{2}\left(n/2\right)}\leq\sqrt{\frac{k}{2k-1}}\leq\sqrt{\frac{2}{3}},
\]
which gives
\begin{equation}
n\leq2\left(2k-1\right)^{10.4}.\label{eq:bound-n-normal}
\end{equation}
Assume now that $\mathcal{D}$ is $m$-periodic, and observe the
$k^{m}$-regular digraph $\mathcal{D}^{'}$ whose vertices are those
of $\mathcal{D}$ and whose edges are the paths of length $m$ in
$\mathcal{D}$. Since $A_{\mathcal{D}^{'}}=A_{\mathcal{D}}^{m}$,
the trivial eigenvalues $e^{2\pi ji/m}k$ of $\mathcal{D}$ become
the eigenvalue $k^{m}$ in $\mathcal{D}'$, which has no other trivial
eigenvalues. This reflects the fact that $\mathcal{D}'$ is a disconnected
digraph with $m$ aperiodic connected components. As $\mathcal{D}'$
is also normal and Ramanujan, (\ref{eq:bound-n-normal}) bounds the
size of each component by $2\left(2k^{m}-1\right)^{10.4}$. All together,
we get 
\begin{equation}
n\leq2m\left(2k^{m}-1\right)^{10.4},\label{eq:bound-normal-periodic}
\end{equation}
so there are only finitely many such graphs.\footnote{An alternate way to handle periodicity is to use \cite[Thm.\ 1]{nilli2004tight}
with $s=m+1$.}
\end{proof}

\begin{rem}
$\left(a\right)$ In §\ref{subsec:Projective-planes} we saw examples
for $2$-periodic normal Ramanujan digraphs with $n\approx2k^{2}$,
which is quite far from the bound (\ref{eq:bound-normal-periodic})
with $m=2$. It seems interesting to ask what is the optimal bound.\\
$\left(b\right)$ In \cite[§5.1]{Lubetzky2017RandomWalks} it is shown
that for any $i\geq1$ there is a walk $W_{i}$ on cells of dimension
$i$ of a complex, such that if $\mathcal{X}$ is a Ramanujan complex
of dimension $d$ then $\mathcal{D}_{W_{i}}\left(\mathcal{X}\right)$
are Ramanujan digraphs for $1\leq i\leq d$. However, no such walk
on vertices (i.e., for $i=0$) is exhibited. Proposition \ref{prop:fin-many-normal}
explains why: it is well known that all geometric operators on vertices
commute with each other (these are ``Hecke operators'' - cf.\ \cite{Lubotzky2005a}).
In particular such an operator commutes with its own transpose, and
therefore induces normal digraphs, which cannot be Ramanujan for an
infinite family by the Proposition.

\end{rem}

\section{\label{sec:almost-normal}Almost-normal digraphs}

In this section we explore almost-normal digraphs, and in particular
almost-normal Ramanujan digraphs. Their main feature, which goes back
to \cite{lubetzky2016cutoff,Lubetzky2017RandomWalks} is the behavior
of powers of their adjacency matrix:
\begin{prop}
\label{prop:bound-powers}Let $\mathcal{D}$ be an $r$-normal, $k$-regular
digraph with $\rho\left(\mathcal{D}\right)=\lambda$. For any $\ell\in\mathbb{N}$,
\begin{align*}
\left\Vert A_{\mathcal{D}}^{\ell}\big|_{L_{0}^{2}}\right\Vert  & \leq{\ell+r-1 \choose r-1}k^{r-1}\lambda^{\ell-r+1}=O\left(\ell^{r-1}\lambda^{\ell}\right).
\end{align*}
\end{prop}

Note that for normal digraphs $r=1$, which gives $\left\Vert A_{\mathcal{D}}^{\ell}\big|_{L_{0}^{2}}\right\Vert \leq\lambda^{\ell}$
as should be. The upshot here is that as long as the ``failure of
normality'' is bounded, only a polynomial price is incurred. This
shows why random walk on almost-normal digraphs is susceptible to
spectral analysis: Let $p_{\ell}$ denote the probability distribution
of the walk at time $\ell$. Assuming for simplicity that $\mathcal{D}$
is aperiodic, so that $L_{0}^{2}=\left\langle \one\right\rangle $,
the distance from equilibrium is
\[
\left\Vert p_{\ell}-\tfrac{\one}{n}\right\Vert =\left\Vert \left(\tfrac{A}{k}\right)^{\ell}p_{0}-\tfrac{\one}{n}\right\Vert =\left\Vert \left(\tfrac{A}{k}\right)^{\ell}\left(p_{0}-\tfrac{\one}{n}\right)\right\Vert \leq\tfrac{1}{k^{\ell}}\left\Vert A^{\ell}\big|_{L_{0}^{2}}\right\Vert =O\left(\ell^{r-1}\left(\tfrac{\lambda}{k}\right)^{\ell}\right),
\]
where we have used $p_{0}-\frac{\one}{n}\in L_{0}^{2}$. In the case
of Ramanujan digraphs $\lambda=\sqrt{k}$, and this gives an almost-optimal
$L^{1}$-cutoff, at time $\log_{k}n+O\left(\log_{k}\log n\right)$
(see \cite[Thm.\ 3.5]{lubetzky2016cutoff} and \cite[Prop.\ 3.1]{Lubetzky2017RandomWalks},
and \cite{Alon2007Nonbacktrackingrandom} for related results). 

An interesting corollary \cite[Thm.\ 2]{Lubetzky2017RandomWalks}
is that in an $r$-normal Ramanujan digraph the sphere of radius $\ell_{0}=\log_{k}n+\left(2r-1\right)\log_{k}\log n$
around any vertex $v$ covers almost all of the graph. Indeed, if
the walk described by $p_{\ell}$ starts at $v_{0}$ then $\supp\left(p_{\ell}\right)=S_{\ell}\left(v_{0}\right)$,
so that 
\[
\frac{n-\left|S_{\ell}\left(v_{0}\right)\right|}{n^{2}}=\sum_{v\notin S_{\ell}\left(v_{0}\right)}\frac{1}{n^{2}}=\left\Vert \left(p_{\ell}-\tfrac{\one}{n}\right)\big|_{V\backslash S_{\ell}\left(v_{0}\right)}\right\Vert ^{2}\leq\left\Vert p_{\ell}-\tfrac{\one}{n}\right\Vert ^{2}=O\left(\frac{\ell^{2r-2}}{k^{\ell}}\right),
\]
and $\ell=\ell_{0}$ yields $\left|S_{\ell_{0}}\left(v_{0}\right)\right|\geq n\left(1-o\left(1\right)\right)$.
This in turn implies a bound of $\left(2+o\left(1\right)\right)\log_{k}\left(n\right)$
on the diameter, since the $\ell_{0}$-spheres around any two vertices
must intersect.

Yet another consequence of almost-normality is a Chernoff bound for
sampling: in \cite{Parzanchevski2018Chernoffboundnon} we show that
if $f$ is a function from the vertices to $\left[-1,1\right]$ with
sum zero, and $v_{1},\ldots,v_{\ell}$ are the vertices visited in
a random walk on an almost-normal directed expander, then
\[
\mathrm{Prob}\left[\tfrac{1}{\ell}\sum_{i=1}^{\ell}f\left(v_{\ell}\right)>\gamma\right]\leq e^{-C\gamma^{2}\ell}
\]
for small enough $\gamma$, where $C$ depends on the expansion and
normality. Using §\ref{subsec:Directed-line-graphs}, this also gives
a similar result for non-backtracking random walk on non-directed
expanders, and via §\ref{subsec:Collision-free-walks-on} to geodesic
walks on high-dimensional expanders.
\begin{proof}[Proof of Proposition \ref{prop:bound-powers}]
By definition, $A$ is unitarily equivalent to a block-diagonal matrix
with blocks of size $r\times r$. The periodic functions on $\mathcal{D}$
correspond to ``trivial'' blocks of size one, and the singular values
of $A^{\ell}\big|_{L_{0}^{2}}$ are the union of the singular values
of the $\ell$-th powers of the remaining, ``nontrivial'' blocks.
Let $B$ be a nontrivial block of size $s\times s$. By Schur decomposition,
we can assume that $B$ is upper triangular, in which case the absolute
values of its diagonal entries are bounded by $\lambda$. In addition,
since $B$ is unitarily equivalent to the restriction of $A$ to some
invariant subspace, all entries of $B$ are bounded by $\left\Vert B\right\Vert _{2}\leq\left\Vert A\right\Vert _{2}=k$,
so that $B$ is entry-wise majorized by 
\[
M_{s,\lambda,k}\overset{{\scriptscriptstyle def}}{=}\left.\left(\begin{matrix}\lambda & k & \cdots & k\\
0 & \lambda & \ddots & \vdots\\
\vdots & \ddots & \ddots & k\\
0 & \cdots & 0 & \lambda
\end{matrix}\right)\ \right\} s.
\]
It follows that $B^{\ell}$ is majorized by $M_{s,\lambda,k}^{\ell}$,
hence using (\ref{eq:Norm2-1-infty}) we have
\[
\left\Vert B^{\ell}\right\Vert _{2}\leq\sqrt{\left\Vert B^{\ell}\right\Vert _{1}\left\Vert B^{\ell}\right\Vert _{\infty}}\leq\sqrt{\left\Vert M_{s,\lambda,k}^{\ell}\right\Vert _{1}\left\Vert M_{s,\lambda,k}^{\ell}\right\Vert _{\infty}}=\left\Vert M_{s,\lambda,k}^{\ell}\right\Vert _{1},
\]
and the latter is just the sum of the first row in $M_{s,\lambda,k}^{\ell}$.
This is maximized for $s=r$, and equals
\[
\sum_{t=0}^{r-1}{r-1 \choose t}{\ell \choose t}k^{t}\lambda^{\ell-t}\leq{\ell+r-1 \choose r-1}k^{r-1}\lambda^{\ell-r+1},
\]
which gives the bound in the Proposition.
\end{proof}
It is natural to ask whether symmetrization turns directed expanders
into expanders, and we suspect that this is true for almost-normal
aperiodic expanders in general. We can show that this is so for the
symmetrization of a high enough power:
\begin{prop}
Let $\mathcal{D}$ be an aperiodic $r$-normal digraph with $\rho\left(\mathcal{D}\right)=\lambda$.
If $\mathcal{G}_{\ell}$ is the symmetrization of the $\ell$-th power
of $\mathcal{D}$, namely $A_{\mathcal{G_{\ell}}}=A_{\mathcal{D}}^{\ell}+\left(A_{\mathcal{D}}^{\ell}\right)^{T}$,
then 
\begin{align*}
\frac{\rho\left(\mathcal{G}_{r-1}\right)}{\deg\mathcal{G}_{r-1}} & =\frac{1}{2}+\frac{\left(r-1\right)^{2}}{2}\cdot\frac{\lambda}{k}+O\left(\left(\tfrac{\lambda}{k}\right)^{2}\right),\qquad\text{and}\\
\frac{\rho\left(\mathcal{G}_{r}\right)}{\deg\mathcal{G}_{r}} & =\frac{r\lambda}{k}+O\left(\left(\tfrac{\lambda}{k}\right)^{2}\right).
\end{align*}
\end{prop}

\begin{proof}
Observe that $\deg\mathcal{G}_{\ell}=2k^{\ell}$. Maintaining the
notations of the previous proof, we have by the same reasoning
\begin{multline*}
\frac{1}{\deg\mathcal{G}_{r-1}}\left\Vert B^{r-1}+B^{*^{r-1}}\right\Vert _{2}\leq\frac{1}{2k^{r-1}}\left\Vert M_{s,\lambda,k}^{r-1}+M_{s,\lambda,k}^{*^{r-1}}\right\Vert _{1}\\
=\frac{1}{2k^{r-1}}\left[\lambda^{r-1}+\sum_{t=0}^{r-1}{r-1 \choose t}^{2}k^{t}\lambda^{r-1-t}\right]=\frac{1}{2}+\frac{\left(r-1\right)^{2}}{2}\cdot\frac{\lambda}{k}+O\left(\left(\tfrac{\lambda}{k}\right)^{2}\right).
\end{multline*}
and the computations for $\mathcal{G}_{r}$ are similar.
\end{proof}
We now prove an Alon-Boppana theorem for almost-normal digraphs:
\begin{thm}
\label{thm:Alon-Boppana-digraphs}Let $k\geq2$ and $m\geq1$. For
any $\varepsilon>0$, there is no infinite almost-normal family of
$k$-regular $m$-periodic digraphs $\mathcal{D}$ with $\rho\left(\mathcal{D}\right)\leq\sqrt{k}-\varepsilon$.
\end{thm}

\begin{proof}
Let $\mathcal{D}$ be an $r$-normal, aperiodic $k$-regular digraph
on $n$ vertices and denote $\lambda=\rho\left(\mathcal{D}\right)$
and $A=A_{\mathcal{D}}$. Let $\mathcal{G}$ be the graph whose adjacency
matrix is $A^{*\ell}A^{\ell}$, for $\ell\geq r$ which will be determined
later on. Namely, $V_{\mathcal{G}}=V_{\mathcal{D}}$, and each edge
in $\mathcal{G}$ corresponds to a $2\ell$-path in $\mathcal{D}$
whose first $\ell$ steps are in accordance with the directions of
the edges of $\mathcal{D}$, and the next $\ell$ steps are in discordance
with them\footnote{In particular, there are $k^{\ell}$ such closed path consisting of
taking some $\ell$-path and then retracing it backwards, so that
one can even take the graph whose adjacency matrix is $A^{*\ell}A^{\ell}-k^{\ell}I$.}. Since $\mathcal{G}$ is $k^{2\ell}$-regular, (\ref{eq:quant-alon-boppana})
gives 
\[
\rho\left(\mathcal{G}\right)\geq2\sqrt{k^{2\ell}-1}\left(1-\frac{2\pi^{2}}{\log_{k^{2\ell}-1}^{2}\left(n/2\right)}\right).
\]
On the other hand, Proposition \ref{prop:bound-powers} gives 
\[
\rho\left(\mathcal{G}\right)=\rho\left(A^{*\ell}A^{\ell}\big|_{L_{0}^{2}}\right)=\left\Vert A^{\ell}\big|_{L_{0}^{2}}\right\Vert ^{2}\leq{\ell+r-1 \choose r-1}^{2}k^{2r-2}\lambda^{2\left(\ell-r+1\right)},
\]
and together we obtain for some $C_{k,r}>0$ 
\begin{align*}
\lambda^{2\left(\ell-r+1\right)} & \geq\frac{2\sqrt{k^{2\ell}-1}}{{\ell+r-1 \choose r-1}^{2}k^{2r-2}}\left(1-\frac{2\pi^{2}}{\log_{k^{2\ell}-1}^{2}\left(n/2\right)}\right)\geq\frac{C_{k,r}k^{\ell-r+1}}{\ell^{2r-2}}\left(1-\frac{8\left(\pi\ell\ln k\right)^{2}}{\ln^{2}\left(n/2\right)}\right)\\
\Longrightarrow\quad\lambda & \geq\sqrt{k}\cdot\sqrt[2\left(\ell-r+1\right)]{\frac{C_{k,r}}{\ell^{2r-2}}\left(1-\frac{8\left(\pi\ell\ln k\right)^{2}}{\ln^{2}\left(n/2\right)}\right)}.
\end{align*}
We finally choose $\ell=\sqrt{\ln\left(n/2\right)}$, obtaining 
\[
\lambda\geq\sqrt{k}\cdot\sqrt[2\left(\ell-r+1\right)]{\frac{C_{k,r}}{\ell^{2r-2}}\left(1-\frac{8\left(\pi\ln k\right)^{2}}{\ln\left(n/2\right)}\right)}\overset{{\scriptscriptstyle n\rightarrow\infty}}{\longrightarrow}\sqrt{k}.
\]
This concludes the aperiodic case, and we leave the general one to
the reader.
\end{proof}

\section{Further exploration}

\subsection{\label{subsec:Universal-Object}Universal Objects}

The universal cover of all $k$-regular graphs is the $k$-regular
tree $\mathcal{T}_{k}$; Ramanujan graphs are those which, save for
the trivial eigenvalues, confine their spectrum to that of their forefather.
It is possible to give an analogous interpretation for Ramanujan digraphs:
consider the $k$-regular \emph{directed} tree $\mathcal{T}_{k}^{\rightleftharpoons}$,
which is obtained by choosing directions for the edges in $\mathcal{T}_{2k}$
to create a $k$-regular digraph. The spectrum of $\mathcal{T}_{k}^{\rightleftharpoons}$
was computed in \cite{Harpe1993spectrumsumgenerators}:
\begin{equation}
\Spec\left(\mathcal{T}_{k}^{\rightleftharpoons}\right)=\left\{ z\in\mathbb{C}\,\middle|\,\left|z\right|\leq\sqrt{k}\right\} ,\label{eq:spcetrum-directed-tree}
\end{equation}
so indeed a $k$-regular digraph is Ramanujan iff its nontrivial spectrum
is contained in that of its ``universal directed cover'' $\mathcal{T}_{k}^{\rightleftharpoons}$.
However, one can also consider other universal objects: for example,
the line digraph $\mathcal{D}_{L}\left(\mathcal{T}_{k+1}\right)$
of the $k+1$-regular tree is a $k$-regular collision-free digraph
which covers all of the digraphs obtained as line graphs of $\left(k\!+\!1\right)$-regular
graphs (see Figure \ref{fig:The-line-digraph} for $k=2$). Its spectrum
is
\[
\Spec\left(\mathcal{D}_{L}\left(\mathcal{T}_{k+1}\right)\right)=\left\{ \pm1\right\} \cup\left\{ z\in\mathbb{C}\,\middle|\,\left|z\right|=\sqrt{k}\right\} ,
\]
and it contains the spectrum of all Ramanujan digraphs of the form
$\mathcal{D}_{L}\left(\mathcal{G}\right)$. It is also $2$-normal:
$L^{2}\left(V_{\mathcal{D}_{L}\left(\mathcal{T}_{k+1}\right)}\right)$
decomposes as an orthogonal direct integral of one and two-dimensional
spaces, each stable under the adjacency operator. Similarly, the digraph
which describes the geodesic walk on the two-dimensional buildings
of type $\widetilde{A}_{2}$ is $3$-normal, and by computations in
\cite{kang2010zeta} its spectrum is 
\begin{equation}
\Spec\left(\mathcal{D}_{W}\left(\widetilde{A}_{2}\right)\right)=\left\{ z\in\mathbb{C}\,\middle|\,\left|z\right|=\sqrt[4]{k}\right\} \cup\left\{ z\in\mathbb{C}\,\middle|\,\left|z\right|=\sqrt{k}\right\} \label{eq:PGL3-spec}
\end{equation}
(see Figure \ref{fig:Cayley_digraph_spec} (right) for a Ramanujan
quotient of this digraph). One can continue to higher dimensions in
this manner - see \cite{kang2016riemann,Lubetzky2017RandomWalks}
for more details.

\subsection{\label{subsec:Universal-Cayley-graphs}Universal Cayley graphs}

For even $k$, the $k$-regular tree $\mathcal{T}_{k}$ is the Cayley
graph of $\mathbf{F}_{k/2}$, the free group on $S=\left\{ x_{1},\ldots,x_{k/2}\right\} $,
with respect to the generating set $S\cup S^{-1}=\left\{ x_{1},\ldots,x_{k/2},x_{1}^{-1},\ldots,x_{k/2}^{-1}\right\} $
(see Figure \ref{fig:The-directed-tree}). In fact, for any subset
$S$ of size $k/2$ in a group $G$, the following are tautologically
equivalent:
\begin{enumerate}
\item $G$ is a free group and $S$ is a free generating set.
\item The Cayley graph $Cay\left(G,S\sqcup S^{-1}\right)$ is a tree\footnote{Here $\sqcup$ indicates disjoint union, so that this is always a
$k$-regular graph.}.
\end{enumerate}
The following, however, is far from a tautology:
\begin{thm}[\cite{Kesten1959}]
For $\frac{k}{2}>1$, $\left(1\right)$ and $\left(2\right)$ above
are equivalent to:
\begin{enumerate}
\item[$(3)$] $\rho\left(A_{Cay\left(G,S\sqcup S^{-1}\right)}\right)=2\sqrt{k-1}$.
\end{enumerate}
\end{thm}

This does not say that $\mathcal{T}_{k}$ is the only $k$-regular
graph with spectral radius $2\sqrt{k-1}$, but rather that among Cayley
graphs it is the only one. In a sense, Keten's result says that the
Ramanujan spectrum characterizes the free group. The analogue for
directed graphs was revealed to be more complex in \cite{Harpe1993spectrumsumgenerators}.
First, observe that $\mathcal{T}_{k}^{\rightleftharpoons}$ is the
Cayley digraph of the free group with respect to the \emph{positive
}generating letters: 
\[
\mathcal{T}_{k}^{\rightleftharpoons}=Cay\left(\mathbf{F}_{k},\left\{ x_{1},\ldots,x_{k}\right\} \right).
\]

\begin{figure}[h]
\begin{centering}
\includegraphics[width=0.6\columnwidth]{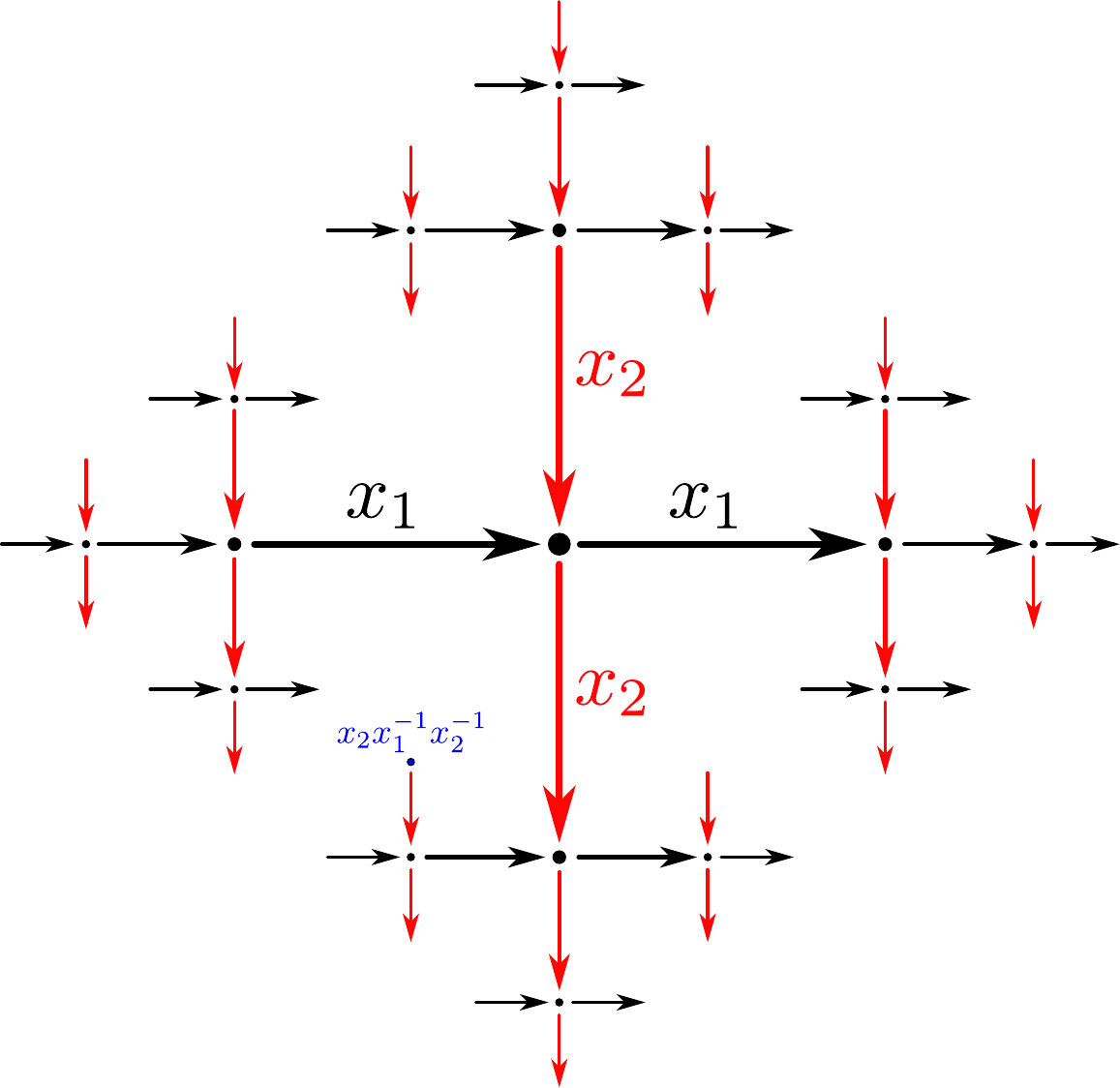}
\par\end{centering}
\caption{\label{fig:The-directed-tree}The directed tree $\mathcal{T}_{2}^{\rightleftharpoons}$
as $Cay\left(\mathbf{F}_{2},\left\{ x_{1},x_{2}\right\} \right)$.}
\end{figure}

As we have said, the spectral radius of $\mathcal{T}_{k}^{\rightleftharpoons}$
is $\sqrt{k}$, but it turns out that it is enough that $S$ generate
a free \emph{semigroup }in order for this to happen:
\begin{thm}[\cite{Harpe1993spectrumsumgenerators}]
\label{thm:Kesten-dlH}Let $S$ be a subset of size $k\geq2$ in
a group $G$. If $S$ generates a free subsemigroup of $G$, then
\[
\rho\left(A_{Cay\left(G,S\right)}\right)=\sqrt{k},
\]
and if $G$ has property (RD)\footnote{Property (RD), which stands for rapid decay, is satisfied both by
hyperbolic groups and by groups of polynomial growth. For its definition
we refer the reader to \cite{jolissaint1990rapidly,Harpe1993spectrumsumgenerators}. } then the converse holds as well.
\end{thm}

For example, small cancellation theory shows that in the surface group
of genus $g\geq2$
\[
S_{g}=\left\langle a_{1},b_{1},\ldots,a_{g},b_{g}\,\middle|\,\left[a_{1},b_{1}\right]\cdot\ldots\cdot\left[a_{g},b_{g}\right]\right\rangle ,
\]
the elements $\left\{ a_{1},b_{1},\ldots,a_{g},b_{g}\right\} $ generate
a free semigroup. Thus, the corresponding Cayley digraph of $S_{g}$
has spectral radius $\sqrt{k}$ even though $S_{g}$ is not free.

\subsection{\label{subsec:Explicit-constructions}Explicit constructions}

For any $k$, \cite{marcus2013interlacing} shows the existence of
infinitely many $k$-regular bipartite Ramanujan graphs, and thus
there exist infinitely many $k$-regular, $2$-periodic, $2$-normal
Ramanujan digraphs, namely their line-digraphs defined in §\ref{subsec:Directed-line-graphs}.
For any prime power $k$, \cite{LPS88,morgenstern1994existence} give
both aperiodic and $2$-periodic $k$-regular Ramanujan digraphs,
as line digraphs of explicit Cayley graphs.

Nevertheless, it is interesting to ask whether Ramanujan digraphs
can be obtained as Cayley digraphs in themselves, and also which groups
$G$ has a generating set $S$ such that $Cay\left(G,S\right)$ is
an almost-normal Ramanujan digraph, as this gives the extremal results
on random walk and diameter mentioned after the statement of Proposition
\ref{prop:bound-powers}.

For $k\in\left\{ 2,3,4,5,7,11,23,59\right\} $, an infinite family
of $k$-regular, $2$-normal Ramanujan digraphs is constructed in
\cite[§5.2]{Parzanchevski2018SuperGoldenGates} as Cayley digraphs
of $PSL_{2}\left(\mathbb{F}_{q}\right)$ and $PGL_{2}\left(\mathbb{F}_{q}\right)$.
Each such family arises from a special arithmetic lattice in the projective
unitary group $PU\left(2\right)$, which acts simply transitively
on the directed edges of the tree $\mathcal{T}_{k+1}$, and whose
torsion subgroup is a group of symmetries of a platonic solid. An
example with $k=4$ is shown in Figure \ref{fig:Cayley_digraph_spec}.

In \cite{Parzanchevski2018Optimalgeneratorsmatrix} we go much further,
showing that for any prime power $q$ and any $d\geq2$ there is an
explicit family of Cayley Ramanujan digraphs on $PSL_{d}\left(\mathbb{F}_{q^{\ell}}\right)$
and $PGL_{d}\left(\mathbb{F}_{q^{\ell}}\right)$ ($\ell\rightarrow\infty$),
which are $k=q^{d-1}$-regular and sharply $d$-normal. As explained
in section §\ref{sec:almost-normal}, this implies that they have
sharp $L^{1}$-cutoff at time $\log_{k}n$, and that their diameter
is bounded by $\left(2+o\left(1\right)\right)\log_{k}\left(n\right)$.
This is quite different from the symmetric case: we have no reason
to suspect that $PSL_{d}\left(\mathbb{F}_{q^{\ell}}\right)$ can be
endowed with a structure of a Ramanujan Cayley graph, for $d\geq3$.
Let us sketch the main ideas: In \cite{cartwright1998family,Lubotzky2005b}
appears an arithmetic lattice $\Gamma$ in a certain division algebra,
which acts simply-transitively on the vertices of the building of
type $\widetilde{A}_{d-1}$ associated with the group $PSL_{d}\left(\mathbb{F}_{q}\left(\left(t\right)\right)\right)$.
This lattice can be enlarged to a lattice $\Gamma<\Gamma'$, which
acts simply-transitively on the edges of color 1 in the same building.
Recall from §\ref{subsec:Collision-free-walks-on} that the geodesic
walk on these edges is $k$-regular and collision-free. We take a
set of generators $S\subseteq\Gamma'$ which induces this walk, and
regard them as elements in the finite group $PSL_{d}\left(\mathbb{F}_{q^{\ell}}\right)$,
which is obtained as a congruence quotient of $\Gamma'$ via strong
approximation. We then invoke the Jacquet\textendash Langlands correspondence
of \cite{badulescu2017global} and the Ramanujan conjecture for function
fields \cite{lafforgue2002chtoucas} to deduce that the nontrivial
spectrum of $S$ on the finite quotient group is contained in the
spectrum of $S$ acting on the building, thus obtaining a Ramanujan
digraph. Finally, sharp $d$-normality follows from \cite[Prop.\ 5.3, 5.4]{Lubetzky2017RandomWalks}.
An example with $d=3$ is shown in Figure \ref{fig:Cayley_digraph_spec},
agreeing with the spectrum of geodesic walk on $\widetilde{A}_{2}$
building shown in (\ref{eq:PGL3-spec}). \newsavebox{\smltwo} \savebox{\smltwo}{$\left\{ \left(\begin{smallmatrix}28 & 4\\
12 & 4
\end{smallmatrix}\right),\left(\begin{smallmatrix}15 & 13\\
10 & 15
\end{smallmatrix}\right),\left(\begin{smallmatrix}6 & 18\\
18 & 13
\end{smallmatrix}\right),\left(\begin{smallmatrix}7 & 3\\
11 & 5
\end{smallmatrix}\right)\right\} $} \newsavebox{\smlthree} \savebox{\smlthree}{$\vphantom{\left(\begin{smallmatrix}\big|\\
0\\
x
\end{smallmatrix}\right)}\left\{ \left(\begin{smallmatrix}0 & 0 & 1\\
0 & x & 0\\
x & x+1 & x
\end{smallmatrix}\right),\left(\begin{smallmatrix}1 & 1 & 1\\
x+1 & x+1 & 1\\
x & 0 & x
\end{smallmatrix}\right),\left(\begin{smallmatrix}1 & 1 & 0\\
x+1 & 1 & 1\\
0 & x+1 & 0
\end{smallmatrix}\right),\left(\begin{smallmatrix}0 & 1 & x\\
x & x+1 & 0\\
1 & x & x
\end{smallmatrix}\right)\right\} $}
\begin{figure}[h]
\begin{centering}
\includegraphics[width=0.99\columnwidth]{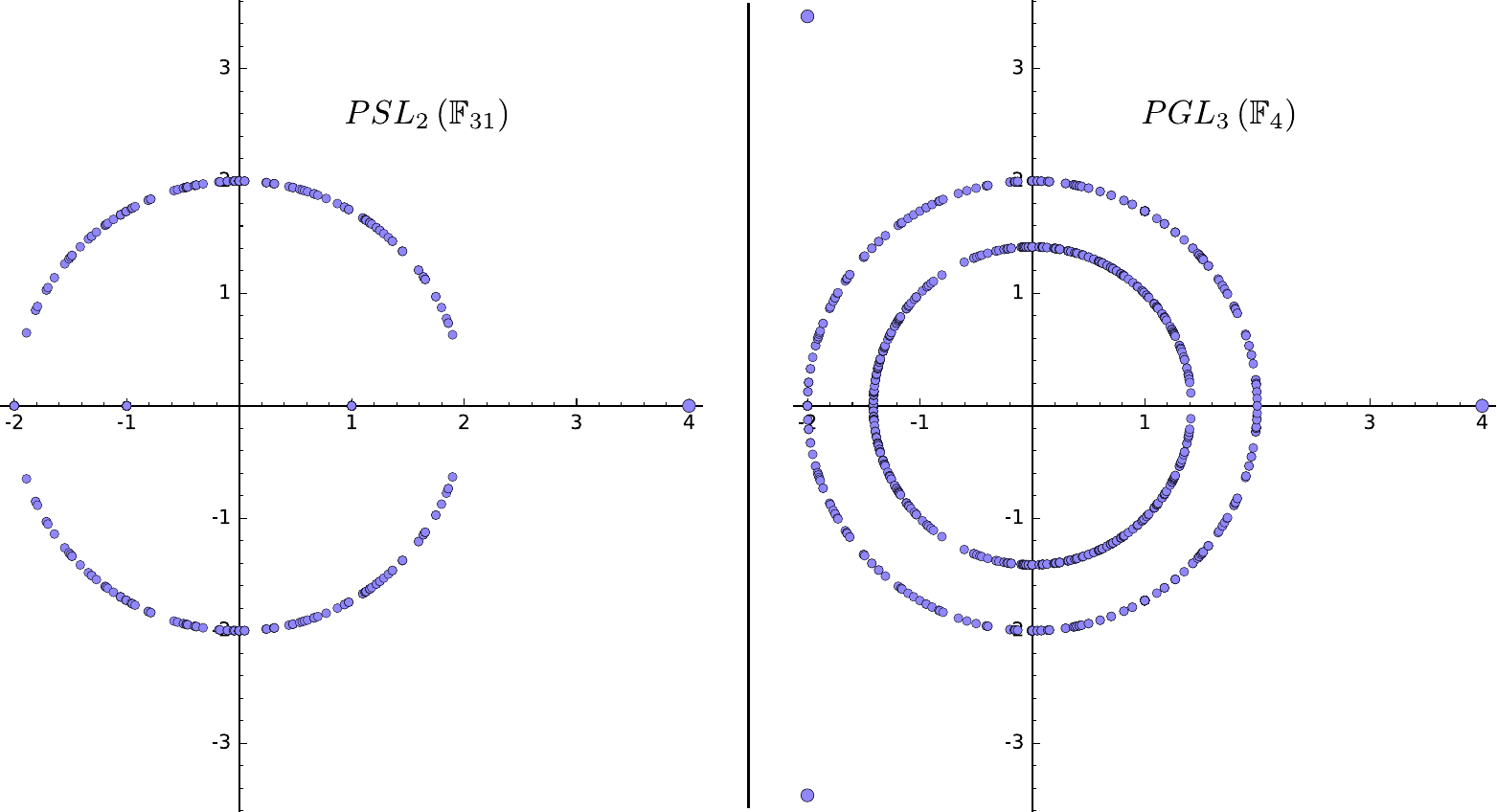} 
\par\end{centering}
\caption{\label{fig:Cayley_digraph_spec}Examples of spectra of Ramanujan Cayley
digraphs: $PSL_{2}\left(\mathbb{F}_{31}\right)$ with generators \usebox{\smltwo}
from \cite{Parzanchevski2018SuperGoldenGates}, and $PGL_{3}\left(\mathbb{F}_{4}\right)$
with generators \usebox{\smlthree} from \cite{Parzanchevski2018Optimalgeneratorsmatrix}.}
\end{figure}

\subsection{\label{subsec:Riemann-Hypothesis}Riemann Hypothesis}

We briefly mention the perspective of zeta functions - for a lengthier
discussion see \cite[§6]{Lubetzky2017RandomWalks} and \cite{kang2010zeta,kang2016riemann,kamber2016lp}.
Ihara \cite{ihara1966discrete} associated with a graph $\mathcal{G}$
a zeta function $\zeta_{\mathcal{G}}\left(u\right)$ which counts
closed cycles in $\mathcal{G}$, in analogy with the Selberg zeta
function of a hyperbolic surface. If $\mathcal{G}$ is \textbf{$\boldsymbol{\left(k\!+\!1\right)}$}-regular,
it is Ramanujan if and only if $\zeta_{\mathcal{G}}\left(u\right)$
satisfies the following ``Riemann hypothesis'': every pole at $\zeta_{\mathcal{G}}\left(k^{-s}\right)$
with $0<\Re s<1$ satisfies $\Re s=\frac{1}{2}$. Indeed, Hashimoto
\cite{hashimoto1989zeta} proved that $\zeta_{\mathcal{G}}\left(u\right)=\det\left(I-u\cdot A_{\mathcal{D}_{L}\left(\mathcal{G}\right)}\right)^{-1}$,
so that (\ref{eq:spec-line-dig}) and (\ref{eq:ram-to-dig-ram}) show
this equivalence (note that the trivial eigenvalues $\pm k$ of $\mathcal{D}_{L}\left(\mathcal{G}\right)$
and the eigenvalues $\pm1$ in (\ref{eq:spec-line-dig}) correspond
to $s=1$ and $s=0$, respectively). For digraphs the story is simpler:
the zeta function $Z_{\mathcal{D}}\left(u\right)$ of a digraph $\mathcal{D}$
(following \cite{bowen1970zeta,hashimoto1989zeta,kotani2000zeta})
is $Z_{\mathcal{D}}\left(u\right)=\prod_{\left[\gamma\right]}\left(1-u^{\ell\left(\gamma\right)}\right)^{-1}$,
where $\gamma$ is a primitive directed cycle of length $\ell\left(\gamma\right)$
in $\mathcal{D}$, and $\left[\gamma\right]$ is the equivalence class
of its cyclic rotations. One then has $Z_{\mathcal{D}}\left(u\right)=\det\left(I-u\cdot A_{\mathcal{D}}\right)^{-1}$,
so that by (\ref{eq:ram-digraph}) a $k$-regular digraph $\mathcal{D}$
is Ramanujan if and only if every pole at $Z_{\mathcal{D}}\left(k^{-s}\right)$
satisfies $\Re s=1$ or $0\leq\Re s\leq\frac{1}{2}$. The fact that
we cannot rule out $s$ with $0<\Re s<\frac{1}{2}$ is demonstrated
by (\ref{eq:PGL3-spec}), for example.

\subsection{\label{subsec:Alon-conjecture}Alon's conjecture}

One of the earliest results on graph expansion is that random regular
graphs are expanders \cite{kolmogorov1967realization,pinsker1973complexity}.
In \cite{Alo86}, Alon conjectured that they are in fact almost Ramanujan.
Namely, for any $\varepsilon>0$ 
\begin{equation}
\mathrm{Prob}\left[\rho\left(\mathcal{G}\right)\leq2\sqrt{k-1}+\varepsilon\right]\overset{{\scriptscriptstyle n\rightarrow\infty}}{\longrightarrow}1\qquad\left(\begin{smallmatrix}\text{where \ensuremath{\mathcal{G}} is a random}\\
\text{\ensuremath{k}-regular graph on \ensuremath{n} vertices}
\end{smallmatrix}\right).
\end{equation}
Alon's conjecture was eventually proved by Friedman \cite{Friedman2008},
and other proofs followed \cite{friedman2014relativized,bordenave2015new}.
While working on the paper \cite{Lubetzky2017RandomWalks}, the author
conjectured that random regular digraphs are almost Ramanujan as well,
in the sense that 
\begin{equation}
\mathrm{Prob}\left[\rho\left(\mathcal{D}\right)\leq\sqrt{k}+\varepsilon\right]\overset{{\scriptscriptstyle n\rightarrow\infty}}{\longrightarrow}1\qquad\left(\begin{smallmatrix}\text{where \ensuremath{\mathcal{D}} is a random}\\
\text{\ensuremath{k}-regular digraph on \ensuremath{n} vertices}
\end{smallmatrix}\right),\label{eq:alon-conjecture-digraphs}
\end{equation}
for any $\varepsilon>0$; and furthermore, that they behave as almost-normal
digraphs, in the sense that the operator norm of their powers is well
behaved as in Proposition \ref{prop:bound-powers}. In joint work
with Doron Puder we tried to extend the methods from \cite{puder2012expansion}
to prove this conjecture, and made partial progress which is described
below. This project was disrupted by the appearance of a solution
on the arXiv:
\begin{thm}[{\cite[Thm.\ 1 with $\delta=\Delta=k$]{coste2017spectral}}]
Statement (\ref{eq:alon-conjecture-digraphs}) is true.
\end{thm}

Since our methods are quite different from the ones in \cite{coste2017spectral},
and might lead to other results (such as understanding of the adjacency-powers),
we sketch them here.

In the seminal paper \cite{broder1987second} the value of $\rho\left(\mathcal{G}\right)$
for a random $k$-regular graph on $n$ vertices is bounded in the
following manner: for even $\ell$, $\rho\left(\mathcal{G}\right)^{\ell}\leq\tr\left(A^{\ell}\right)-k^{\ell}$,
and $\tr\left(A^{\ell}\right)$ equals the number of closed paths
of length $\ell$ in $\mathcal{G}$. In the permutation model for
$\mathcal{G}$ (see \cite{broder1987second,wormald1999models,puder2012expansion})
each path of length $\ell$ is determined by a starting vertex, and
a word $\omega$ of length $\ell$ in $S=\left\{ x_{1}^{\pm1},\ldots,x_{k/2}^{\pm1}\right\} $.
If $\omega$ is trivial as an element of $\mathbf{F}_{k/2}$, this
path is completely backtracking in every instance of $\mathcal{G}$,
and in particular closed. Denoting $p_{\omega}=\mathrm{Prob}\left({\text{a path in \ensuremath{\mathcal{G}} which starts at \ensuremath{v}}\atop \text{ and is labeled by \ensuremath{\omega} ends at \ensuremath{v}}}\right)-\frac{1}{n}$,
one obtains 
\[
\mathbb{E}\left(\rho\left(\mathcal{G}\right)^{\ell}\right)\leq\mathbb{E}\left(\tr\left(A^{\ell}\right)\right)-k^{\ell}=n\sum_{\omega\in S^{\ell}}p_{\omega},
\]
and each trivial $\omega$ contributes $p_{\omega}=1-\frac{1}{n}$.
In \cite{broder1987second} it is shown that $p_{\omega}$ is small
for words which are not trivial or proper powers in $\mathbf{F}_{k/2}$,
and the number of trivial and power words is bounded, giving a bound
on $\mathbb{E}(\rho\left(\mathcal{G}\right)^{\ell})$. An appropriate
choice of $\ell$ then implies $\rho\left(\mathcal{G}\right)\leq3k^{3/4}$
a.a.s.\ as $n\rightarrow\infty$. In \cite{Pud13,puder2012measure}
it is shown that $p_{\omega}$ depends on the so-called \emph{primitivity
rank} of $\omega$, and in \cite{puder2012expansion} this is made
qualitatively precise, and words of each primitivity rank are counted,
leading to $\rho\left(\mathcal{G}\right)\leq2\sqrt{k-1}+1$ a.a.s. 

Now, let $\mathcal{D}$ be a $k$-regular digraph on $n$ vertices,
so that $A_{\mathcal{D}}$ is simply the sum of $k$ independent $n\times n$
permutation matrices. We cannot use $\tr\left(A^{\ell}\right)$ directly
to bound $\rho\left(\mathcal{G}\right)$, since $A$ is not normal
anymore. Instead, denoting $A_{0}=A\big|_{L_{0}^{2}}$ we use Gelfand's
formula:
\[
\sqrt[2\ell]{\rho\left(A_{0}^{*\ell}A_{0}^{\ell}\right)}=\sqrt[\ell]{\left\Vert A_{0}^{\ell}\right\Vert }\:\underset{{\scriptscriptstyle \ell\rightarrow\infty}}{\text{{\scalebox{1.5}[1]{\ensuremath{\searrow}}}}}\:\rho\left(A_{0}\right),
\]
and to bound $\rho\left(A_{0}^{*\ell}A_{0}^{\ell}\right)$ we study
$\tr\left((A^{*\ell}A^{\ell})^{t}\right)$. The entries of $\left(A^{*\ell}A^{\ell}\right)^{t}$
correspond to ``$\ell$-alternating words'' of length $2\ell t$:
words in $S=\left\{ x_{1}^{\pm},\ldots,x_{k}^{\pm}\right\} $ which
are composed of alternating sequences of $\ell$ negative letters
followed by $\ell$ positive ones. Given a starting vertex, each such
word translates to a path in $\mathcal{D}$, where negative letters
indicate crossing a directed edge in the ``wrong'' direction. Again
$p_{\omega}$ is the probability that this path is closed, so that
\[
\rho\left(\left(A_{0}^{*\ell}A_{0}^{\ell}\right)^{t}\right)\leq\tr\left(\left(A^{*\ell}A^{\ell}\right)^{t}\right)-k^{2\ell t}=n\cdot\sum_{\omega\in\left(\smash{S_{+}^{\ell}\times S_{-}^{\ell}}\right)^{t}}p_{w}.
\]
Now, $\mathbb{E}\left(\tr\left(\left(A^{*\ell}A^{\ell}\right)^{t}\right)\right)$
can be bounded similarly to \cite{puder2012expansion}, this time
by counting $\ell$-alternating words of each primitivity rank, and
choosing both $\ell$ and $t$ carefully. We discovered that already
from $\ell=2$ one obtains the bound $\rho\left(\mathcal{D}\right)\leq\sqrt{2k}+\varepsilon$
a.a.s., and we expect that as $\ell$ goes to infinity one should
recover (\ref{eq:alon-conjecture-digraphs}) up to an additive constant.
As remarked above, this analysis goes through the spectral norm of
$A^{\ell}$, so it might lead to other interesting results on $\mathcal{D}$.

\section{\label{sec:Questions}Questions}

\setlist[enumerate]{leftmargin=*,widest=0}
\begin{enumerate}
\item A \emph{non-regular }graph $\mathcal{G}$ is said to be Ramanujan
if its nontrivial spectrum is contained in the $L^{2}$-spectrum of
its universal cover (which is a non-regular tree). This definition
is justified both by the extended Alon-Boppana theorem \cite{greenberg1995spectrum,grigorchuk1999asymptotic}
and by the behavior of random covers \cite{friedman2003relative,puder2012expansion,friedman2014relativized,brito2018spectral}.
What is the appropriate definition of a non-regular Ramanujan digraph?
\item Can standard results on expanders (such as the Cheeger inequalities
and the expander mixing lemma) be extended to almost-normal directed
expanders?
\item Does symmetrization turn a family of almost-normal directed expanders
into a family of expanders?
\item Are there infinite almost-normal families of non-periodic $k$-regular
Ramanujan digraphs for $k$ which is not a prime power?
\item Is there an almost-normal family of $k$-regular digraphs with $n\rightarrow\infty$
whose nontrivial spectrum is contained in the circle $\{z\,|\,\left|z\right|=\sqrt{k}\}$?
\item Almost-normality is an ``algebraic'' phenomenon: it originates from
representation theory in \cite{Lubetzky2017RandomWalks}, and from
the special structure of line-digraphs in \cite{lubetzky2016cutoff}.
There seems to be no reason that random models will have this property,
or that it will be stable under perturbations. Is there a more flexible
definition of almost-normality, which still gives a theorem in the
spirit of Proposition \ref{prop:bound-powers}?
\item How important is almost-normality? Is there a family of $k$-regular
Ramanujan digraphs which behave like amenable graphs in terms of expansion?
\item In the infinite case we can even ask whether a \emph{single }digraph
is almost-normal, meaning that its adjacency operator is unitarily
equivalent to a direct integral of linear operators of bounded dimension.
Such digraphs can be obtained by taking line-digraphs of infinite
graphs, or more generally $\mathcal{D}_{W}\left(\mathcal{B}\right)$
where $W$ is some walk on a building $\mathcal{B}$ (or an infinite
quotient of a building). To what extent does the spectral theory of
infinite symmetric graphs carries over to almost-normal digraphs?
\end{enumerate}
\bibliographystyle{amsalpha}
\bibliography{/home/ori/Dropbox/Math/mybib}

\begin{flushleft}
\noun{Einstein Institute of Mathematics, Hebrew University of Jerusalem,
Israel.}\texttt{}~\\
\texttt{parzan@math.huji.ac.il.}
\par\end{flushleft}
\end{document}